\documentclass[a4paper,11pt]{amsart}
\usepackage{amssymb,amscd}
\usepackage[pdftex]{graphicx}
\usepackage{enumitem}
\usepackage[capitalize]{cleveref}
\numberwithin{equation}{section}

\newtheorem{cor}[equation]{Corollary}
\newtheorem{lem}[equation]{Lemma}

\newtheorem{thm}[equation]{Theorem}
\theoremstyle{definition}
\newtheorem{dfn}[equation]{Definition}

\newtheorem{rem}[equation]{Remark}

\def\R{\mathbb R}

\def\Z{\mathbb Z}

\def\ve{\varepsilon}
\def\vf{\varphi}

\DeclareMathOperator{\arsinh}{arsinh}
\DeclareMathOperator{\artanh}{artanh}
\DeclareMathOperator{\ess}{ess}
\DeclareMathOperator{\inrad}{inrad}
\DeclareMathOperator{\supp}{supp}
\DeclareMathOperator{\sys}{sys}
\DeclareMathOperator{\Sys}{Sys}

\begin{document}


\title{Small eigenvalues of surfaces - old and new}
\author{Werner Ballmann}
\address
{WB: Max Planck Institute for Mathematics,
Vivatsgasse 7, 53111 Bonn
and Hausdorff Center for Mathematics,
Endenicher Allee 60, 53115 Bonn}
\email{hwbllmnn\@@mpim-bonn.mpg.de}
\author{Henrik Matthiesen}
\address
{HM: Max Planck Institute for Mathematics,
Vivatsgasse 7, 53111 Bonn,
and Hausdorff Center for Mathematics,
Endenicher Allee 60, 53115 Bonn}
\email{hematt\@@mpim-bonn.mpg.de}
\author{Sugata Mondal}
\address
{SM: Indiana University,
Rawles Hall, 831 E 3rd Street,
Bloomington, Indiana}
\email{sumondal\@@iu.edu}

\thanks{\emph{Acknowledgments.}
We would like to thank Bram Petri and Federica Fanoni for pointing out the results of \cite{Pa} to us.
We are grateful to the Max Planck Institute for Mathematics, the Hausdorff Center for Mathematics, and Indiana University at Bloomington for their support and hospitality.}

\date{\today}


\begin{abstract}
We discuss our recent work on small eigenvalues of surfaces.
As an introduction, we present and extend some of the by now classical work of Buser and Randol
and explain novel ideas from articles of S\'evennec, Otal, and Otal-Rosas which are of importance in our line of thought.
\end{abstract}

\maketitle

\tableofcontents

\section{Introduction}
\label{intro}
Every conformal class of a connected surface
S contains a complete Riemannian metric with constant curvature $K$.
This metric is unique up to scale.
We say that it is \emph{hyperbolic} if $K=-1$. 
In particular, any conformal class on the orientable closed surface $S_\gamma$ of genus $\gamma\ge2$ contains a unique hyperbolic metric.
For this reason, hyperbolic metrics on $S_\gamma$ play a specific role.

Traditionally, an eigenvalue of a hyperbolic metric on $S_\gamma$ is said to be \emph{small} if it is below $1/4$.
This terminology may have occurred for the first time in \cite[page 386]{Hu}.
The importance of the number $1/4$ stems from the fact that it is the bottom of the spectrum of the hyperbolic plane.

In \cite[1972]{McK}, McKean stated erroneously that hyperbolic metrics on $S_\gamma$ do not carry non-trivial small eigenvalues.
This was corrected by Randol in \cite[1974]{Ra} who showed the existence of arbitrarily many small eigenvalues.

\begin{thm}[Randol]\label{iran}
For each hyperbolic metric on $S_\gamma$ and natural number $n\ge1$,
there is a finite Riemannian covering $\tilde S\to S_\gamma$ such that $\tilde S$ has at least $n$ eigenvalues in $[0,1/4)$.
\end{thm}

The proof of Randol uses Selberg's trace formula. 
Later, Buser observed that geometric methods lead to an elementary construction of hyperbolic metrics on $S_\gamma$ with many arbitrarily small eigenvalues \cite[1977]{Bu1}.

\begin{thm}[Buser]\label{ibus}
For any $\gamma\ge2$ and $\ve>0$, there are hyperbolic metrics on $S_\gamma$ with $2\gamma-2$ eigenvalues in $[0,\ve)$.
\end{thm}

The construction of Buser relies on a pairs of pants decomposition of $S_\gamma$, where the boundary geodesics of the hyperbolic metrics on the pairs of pants are sufficiently short.
Schoen--Wolpert--Yau \cite[1980]{SWY} generalized Buser's result and estimated low eigenvalues in terms of decompositions of $S_\gamma$,
where eigenvalues are counted according to their multiplicity.

\begin{thm}[Schoen--Wolpert--Yau]\label{iswy}
For any Riemannian metric on $S_\gamma$ with curvature $-1\le K\le-k<0$, the $i$-th eigenvalue satisfies
\begin{align*}
  \alpha_1k^{3/2}\ell_i\le\lambda_i\le \alpha_2\ell_i \quad\text{for}\quad 0<i<2\gamma-2 \quad\text{and}\quad \alpha_1k\le\lambda_{2\gamma-2}\le \alpha_2,
\end{align*}
where $\alpha_1,\alpha_2>0$ depend only on $\gamma$ and $\ell_i$ is the minimal possible sum of the lengths of simple closed geodesics in $S_\gamma$ which cut $S_\gamma$ into $i+1$ pieces.
\end{thm}

Dodziuk--Pignataro--Randol-Sullivan \cite[1987]{DPRS} extended \cref{iswy} to complete hyperbolic metrics on the orientable surfaces $S_{\gamma,p}$ of genus $\gamma$ with $p$ punctures (in the sense of \cref{secsu} below).

Whereas the left inequalities in \cref{iswy} show the necessity of short simple closed geodesics for the existence of small eigenvalues,
the inequality for $\lambda_{2\gamma-2}$ on the right indicates that this eigenvalue plays a different role.
Indeed, Schmutz showed that $\lambda_2\ge1/4$ for any hyperbolic metric on $S_2$ \cite[1991]{S2}.
Furthermore, Buser showed in \cite[1992]{Bu2} that $\lambda_{2\gamma-2}>\alpha>0$ for any hyperbolic metric on $S_\gamma$, where $\alpha$ does not depend on $\gamma$.

Inspired by the above results and presumably also by their previous estimates $\lambda_{4\gamma-2}>1/4$ (Buser \cite[1977]{Bu1}) and $\lambda_{4\gamma-3}>1/4$ (Schmutz \cite[1990]{S1}),
Buser and Schmutz conjectured that $\lambda_{2\gamma-2}\ge1/4$ for any hyperbolic metric on $S_\gamma$.
The development so far is what we refer to as old in our title,
and our presentation of it is trimmed towards our needs.
The new development starts with the work of Otal and Rosas, who proved the following strengthened version of the Buser-Schmutz conjecture in \cite[2009]{OR},
using ideas from S\'evennec \cite[2002]{Se} and Otal \cite[2008]{Ot}.

\begin{thm}[Otal-Rosas]\label{ior}
For any real analytic Riemannian metric on $S_\gamma$ with negative curvature, we have $\lambda_{2\gamma-2}>\lambda_0(\tilde S)$, where $\lambda_0(\tilde S)$ denotes the bottom of the spectrum of the universal covering surface $\tilde S$ of $S$, endowed with the lifted metric. 
\end{thm}

In his Bachelor thesis \cite[2013]{Mat}, the second named author showed that the assumption on the curvature is superfluous.
In his PhD thesis, the third named author showed that, for hyperbolic metrics on $S_\gamma$, the inequality in \cref{ior} can be sharpened to $\lambda_{2\gamma-2}>\lambda_0(\tilde S)+\delta$, where $\delta>0$ is a constant depending on $\gamma$ and the systole of the metric; see  \cite[2014]{Mo}.

Obviously, the assumption on the real analyticity of the Riemannian metric in \cref{ior} is unpleasant.
At the expense of the strictness of the  inequality, the assumption can be removed, by the density of the space of real analytic Riemannian metrics inside the space of smooth ones.
In \cite[Question 2]{OR}, Otal and Rosas speculate about the possibility of removing the assumption, keeping the strictness of the inequality.

The starting point of our joint work is this last question on the real analyticity of the Riemannian metric.
We could show that the assumption can indeed be removed in the case of closed surfaces with negative Euler characteristic \cite[2016]{BMM1}.
Later we could show the inequality even in the general case of complete Riemannian metrics on surfaces $S$ of finite type (in the sense of \cref{secsu}) \cite[2017]{BMM2}. 
In addition, our inequality $\lambda_{2\gamma-2}>\Lambda(S)$ in these papers improves the inequality of Otal and Rosas in \cite{OR} and also of the third named author in \cite{Mo}.
The new invariant $\Lambda(S)$, the \emph{analytic systole} of the Riemannian metric of the surface, satisfies $\Lambda(S)\ge\lambda_0(\tilde S)$.
Furthermore, for a large class of surfaces, including compact surfaces of negative Euler characteristic, the inequality is strict.
This is the main result of our paper \cite{BMM3}.
A mayor part of this article, Sections \ref{secns} and \ref{secas}, is devoted to explaining the main ideas behind the new developments and to formulate our main results. 

After fixing some notation and discussing some preliminaries in Sections \ref{secsu} and \ref{secan}, we present extensions of \cref{iran} and \cref{ibus} in \cref{secbr}.
In particular, we obtain a quantitative generalization of \cref{iran} by elementary geometric arguments.
In \cref{secns}, we discuss the ideas of S\'evennec, Otal, and Otal-Rosas which were important in our work in \cite{BMM1,BMM2}.

The rather long \cref{secas} discusses our results from \cite{BMM1}-\cite{BMM3}.
First, in \cref{secas_small}, we describe the arguments needed to extend the ideas from \cite{OR} to get the improved bound on the number of small eigenvalues.
We then proceed in \cref{secas_quant} and \cref{secas_qual} to the bounds for the analytic systole obtained in \cite{BMM3}.
Since the arguments for the qualitative bounds of $\Lambda(S)$ are quite involved, we restrict the discussion to compact surfaces, which reduces the technicalities substantially.
 
\section{Preliminaries on surfaces}
\label{secsu}
We say that a surface $S$ is of \emph{finite type} if its boundary $\partial S$ is compact (possibly empty) and its Euler characteristic $\chi(S)>-\infty$.
A surface is of finite type if it is diffeomorphic to a closed surface with a finite number of pairwise disjoint points and open disks removed, so called \emph{punctures} and \emph{holes}, respectively.
A connected surface $S$ of finite type can be uniquely written as a closed orientable surface $S_\gamma$ of genus $\gamma\ge0$ with $p\ge0$ punctures, $q\ge0$ holes, and $0\le r\le2$ embedded M\"obius bands.
Extending our notation, we write $S_{\gamma, p}$, $S_{\gamma, p, q}$, and $S_{\gamma, p, q, r}$ for $S_\gamma$ with $p$ punctures, with $p$ punctures and $q$ holes, and with $p$ punctures, $q$ holes and $r$ embedded M\"obius bands, respectively.
We call $\gamma$ the \emph{genus} of $S_{\gamma,p,q,r}$.
We have \[\chi(S_{\gamma,p,q,r})=2-2\gamma-p-q-r.\]

\begin{rem}\label{pp}
A surface $S$ of finite type with $\chi(S)<0$ admits decompositions into pairs of pants, that is, into building blocks $P$ of the following type:
\begin{enumerate}[label=\arabic*)]
\item\label{3h} a sphere with three holes;
\item\label{2h} a sphere with two holes and one puncture;
\item\label{1h} a sphere with one hole and two punctures; 
\item\label{0h} a sphere with three punctures;
\item\label{1c} a sphere with two holes and an embedded M\"obius band.
\end{enumerate}
Each of these building blocks of $S$ has Euler characteristic $-1$ and circles as boundary components.
Hence $S$ is built of $-\chi(S)$ such blocks, where we need at most $2$ of the kind \ref{1c}
and where a block $P$ of type \ref{0h} occurs if and only if $S=P$. 
\end{rem} 

A Riemannian metric on a surface is called a \emph{hyperbolic metric} if it has constant curvature $-1$.
A surface together with a hyperbolic metric will be called a \emph{hyperbolic surface}.
A connected surface $S$ of finite type admits complete hyperbolic metrics of finite area with closed geodesics as boundary circles if and only if $\chi(S)<0$.
That is, excluded are sphere, projective plane, torus, Klein bottle, disk, and annulus.

\begin{rem}
With respect to any complete hyperbolic metric with closed geodesics as boundary circles, (neighborhoods of) the ends $U$ of a surface $S$ of finite type with $\chi(S)<0$ are of one of the following two types:
\begin{enumerate}[label=\arabic*)]
\item\label{cusp} Cusps: $U=C_\ell=[0,\infty)\times\R/\ell\Z$ with metric $dr^2+e^{-2r}ds^2$.
\item\label{funnel} Funnels: $U=F_\ell=[0,\infty)\times\R/\ell\Z$ with metric $dr^2+\cosh(r)^2ds^2$.
\end{enumerate}
The geodesics $\{s=\rm{const}\}$ on cusps and funnels will be called \emph{outgoing},
the geodesic $\{r=0\}$ of a funnel will be called the \emph{base geodesic} of the funnel.
\end{rem}

On building blocks $P$ of type \ref{pp}(\ref{3h}, \ref{pp}(\ref{2h}, and \ref{pp}(\ref{1h}, the family of hyperbolic metrics on $P$, with finite area and with closed geodesics as boundary circles, may be parametrized by the lengths of the boundary circles, respectively.
There is exactly one complete hyperbolic metric of finite area on a building block of type \ref{pp}(\ref{0h}.
On building blocks of type \ref{pp}(\ref{1c}, there is a three-parameter family of hyperbolic metrics with closed geodesics as boundary circles, where the lengths of the two boundary circles and the soul of the M\"obius band can be chosen as parameters in $(0,\infty)$.

Finite area is equivalent to the requirement that all the ends of $P$ are cusps.
However, it is also possible to have an arbitrary subset of the ends of $P$ to consist of funnels instead, where the lengths of the bases of the funnels may serve as additional parameters for the family of complete hyperbolic metrics.

\begin{thm}[Collar theorem]\label{coll}
For any complete hyperbolic surface $S$ of finite type and any two-sided simple closed geodesic $c$ in $S$ of length $\ell$,
the neighborhood of width $\rho=\arsinh(1/\sinh(\ell/2))$ about $c$ in $S$ is isometric to $(-\rho,\rho)\times\R/\ell\Z$ with Riemannian metric $dr^2+\cosh(r)^2ds^2$.
\end{thm}

\begin{proof}
For any complete hyperbolic metric on a building block $P$ of type \ref{pp}(\ref{3h}, \ref{pp}(\ref{2h}, and \ref{pp}(\ref{1h}, respectively, with finite area and with closed geodesics as boundary circles, the collar of width $\rho$ about a boundary circle $c$ of length $\ell$ is isometric to $[0,\rho)\times\R/\ell\Z$ with Riemannian metric $dr^2+\cosh(r)^2ds^2$ and is disjoint from the corresponding collar about any other boundary circle of $P$, by \cite[Propositions 3.1.8 and 4.4.4]{Bu2}.

Suppose first that $S$ is orientable and let $c$ be as in \cref{coll}.
Then we may start a pair of pants decomposition with $c$ as first cut and conclude the assertion from what we just said.

Suppose now that $S$ is not orientable.
Since $c$ is two-sided, the preimage of $c$ under the orientation covering $S_o\to S$ consists of two simple closed geodesics $c_1$ and $c_2$ of length $\ell$.
We may start a pair of pants decomposition of $S_o$ with $c_1$ and $c_2$ as first cuts.
By what we said in the beginning of the proof, the neighborhoods of radius $\rho$ about these are disjoint.
\end{proof}

For a complete and connected Riemannian surface $S$ of finite type without boundary, denote by $\sys(S)$ the \emph{systole of $S$} and by $\sys^*(S)$ the length of a shortest two-sided non-separating simple closed geodesic of $S$.
If $S$ admits complete hyperbolic metrics, denote by $\Sys(S)$ the maximal possible systole among complete hyperbolic metrics on $S$ and by $\Sys^*(S)$ the maximal possible length of a shortest two-sided non-separating simple closed geodesic among hyperbolic metrics on $S$.

Note that on $S_\gamma=S_{\gamma,0}$ and $S_{\gamma,1}$, a simple closed curve is separating if and only if it is homologically trivial.
Hence the following result is a consequence of Parlier's \cite[Theorem 1.1]{Pa}.

\begin{thm}[Parlier]\label{parl}
Among all complete hyperbolic metrics of finite area on a closed orientable surface $S$ with at most one puncture, $\Sys(S)$ and $\Sys^*(S)$ are achieved.
Furthermore, a complete hyperbolic metric achieving $\Sys(S)$ also achieves $\Sys^*(S)$.\end{thm}

\section{Preliminaries on spectral theory}
\label{secan}
Let $M$ be a Riemannian manifold, possibly not complete and possibly with non-empty boundary $\partial M$.
Denote by $C^k(M)$ the space of $C^k$-functions on $M$,
by $C^k_{c}(M)\subseteq C^k(M)$ the space of $C^k$-functions on $M$ with compact support,
and by $C^k_{cc}(M)\subseteq C^k_c(M)$ the space of $C^k$-functions on $M$ with compact support in the interior $\mathring M=M\setminus\partial M$ of $M$, respectively.
Denote by $L^2(M)$ the space of (equivalence classes of) square-integrable measurable functions on $M$.
For integers $k\ge0$, let $H^k(M)$ be the Sobolev space of functions in $L^2(M)$ which have, for all $0\le j\le k$, square-integrable $j$-th derivative $\nabla^{j}f$ in the sense of distributions, that is, tested against functions from $C^\infty_{cc}(M)$.
Let $H^k_0(M)$ be the closure of $C^\infty_{cc}(M)$ in $H^k(M)$.

Denote by $\Delta$ the Laplace operator of $M$ and by $\nu$ the outward normal field of $M$ along $\partial M$.
For the following result, see for example \cite[page 85]{T2}.

\begin{thm}
If $M$ is complete, then the Laplacian $\Delta$ with domains
\begin{align*}
  \mathcal{D}_0 = \{\vf\in C^\infty_c(M) \mid \vf|_{\partial M}=0\}
  \quad\text{and}\quad
  \mathcal{D}_N = \{\vf\in C^\infty_c(M) \mid \nabla_\nu\vf=0\}
\end{align*}
is essentially self-adjoint in $L^2(M)$.
\end{thm}

The corresponding closures of $\Delta$ will be called the \emph{Dirichlet} and \emph{Neumann extension} of $\Delta$, respectively.
In the same vein, we will speak of Dirichlet and Neumann spectrum or eigenvalues of $M$.
Note that these notions coincide with the usual ones if the boundary of $M$ is empty.

For a non-vanishing $\vf\in C^\infty_c(M)$,
\begin{align}\label{ray}
  R(\vf) = \frac{\int_M|\nabla\vf|^2}{\int_M\vf^2}
\end{align}
is called the \emph{Rayleigh quotient} of $\vf$.
The real numbers
\begin{align}\label{bot}
  \lambda_0(M) = \inf_{\substack{\vf\in C^\infty_{cc}(M)\\ \vf\ne0}} R(\vf)
  \quad\text{and}\quad
  \lambda_N(M) = \inf_{\substack{\vf\in C^\infty_c(M)\\ \vf\ne0}} R(\vf)
\end{align}
are equal to the minimum of the Dirichlet and Neumann spectrum of $M$ and are, therefore, called the \emph{bottom of the Dirichlet} and \emph{Neumann spectrum of $M$}, respectively.
If $\partial M=\emptyset$, then $\lambda_0(M)=\lambda_N(M)$.
If $M$ is closed, that is, compact and connected without boundary, $\lambda_0(M)=0$.
If $M$ is compact and connected with non-empty boundary, $\lambda_0(M)$ is the smallest Dirichlet eigenvalue of $M$ and $\lambda_N(M)=0$. 
If $M$ is connected and $\tilde{M}\to M$ is a normal Riemannian covering with amenable group of covering transformations, then $\lambda_0(\tilde M)=\lambda_0(M)$; see \cite{Br2} and, for the generality of the statement here and a more elementary argument, see also \cite{BMP}.

The \emph{essential spectrum} $\sigma_{\ess}(A)$ of a self-adjoint operator $A$ on a Hilbert space $H$ consists of all $\lambda\in\R$ such that $A-\lambda$ is not a Fredholm operator.
The essential spectrum of $A$ is a closed subset of the spectrum $\sigma(A)$ of $A$.
The complement $\sigma_d(A)=\sigma(A)\setminus\sigma_{\ess}(A)$, the \emph{discrete spectrum of $A$}, is a discrete subset of $\R$ and consists of eigenvalues of $A$ of finite multiplicity.

If $M$ is a complete Riemannian manifold with compact boundary (possibly empty),
then the  essential Dirichlet and Neumann spectra of $M$ coincide and their infimum is given by
\begin{equation}\label{ess}
  \lambda_{\ess}(M) = \sup_K\inf\{R(\vf)\mid \vf\in C^\infty_c(M\setminus K), \vf\ne0\},
\end{equation}
where $K$ runs over all compact subsets of $M$;
compare with \cite[Theorem 14.4]{HS}, where \eqref{ess} is shown for Schr\"odinger operators on Euclidean spaces.
By \eqref{ess}, if $M$ is compact, then $\lambda_{\ess}(M)=\infty$, that is, the essential Dirichlet and Neumann spectra of $M$ are empty.

Since a basis of neighborhoods of any end of a surface of finite type may be chosen to consist of annuli, we have
\begin{align}\label{ess0}
  \lambda_{\ess}(S)\ge\lambda_0(\tilde S)
\end{align}
for any complete Riemannian surface $S$ of finite type,
where $\tilde S\to S$ denotes the universal covering of $S$ and $\tilde S$ is endowed with the lifted Riemannian metric; see \cite[Proposition 3.9]{BMM2}.

\begin{rem}
For any complete hyperbolic surface of finite type, we get that $\lambda_{\ess}(S)\ge1/4$.
On the other hand, any surface $S$ of infinite type admits complete hyperbolic metrics with corresponding $\lambda_{\ess}(S)=0$.
\end{rem}

\section{Theorems \ref{iran} and \ref{ibus} revisited}
\label{secbr}
Mainly relying on the original argument of Buser, we discuss the following extended version of Buser's \cref{ibus}.

\begin{thm}\label{bus}
Let $S$ be a connected surface of finite type with Euler characteristic $\chi(S)<0$, and let $\ve\in(0,1/4]$.
Then $S$ carries complete hyperbolic metrics, with closed geodesics as boundary circles if $\partial S\ne\emptyset$, with $-\chi(S)$ Dirichlet eigenvalues in $[0,\ve)$.
\end{thm}

If $\partial S=\emptyset$, the Dirichlet eigenvalues are the usual eigenvalues of $S$.

For $\ve=1/4$, \cref{bus} also follows from the main result of \cite{DPRS}, at least in the case where $\partial S=\emptyset$.
Nevertheless, it seems appropriate to us to include the proof of \cref{bus}
since the argument is nice and short and the proof of \cref{ran} uses a variation of it.

\begin{proof}[Proof of \cref{bus}]
Choose a decomposition of $S$ into pairs of pants as in \cref{pp}.
For each building block $P$ of the decomposition, choose a hyperbolic metric on $P$ with closed geodesics as boundary circles
and with cusps and funnels around the punctures.
Independently of the chosen hyperbolic metric, the area of $P$ minus its funnels (if funnels occur) is $2\pi$, by the Gauss-Bonnet formula.
The conditions on the hyperbolic metrics on the different $P$ are that the lengths of the boundary circles are sufficiently small and fitting according to the decomposition of $S$.
For the funnels, the corresponding lengths of their base geodesics should also be sufficiently small.
The meaning of sufficiently small will become clear from the following construction of test functions.

On each building block $P$, we consider the function $\vf_P$ which is equal to $1$ on the set $Q$ of points of $P$ of distance at least $1$ from the boundary circles and funnels of $P$, vanishes on the boundary circles and funnels of $P$, and decays linearly from $1$ to $0$ along the normal geodesic segments in between.
We arrive at the first condition which is that the neighborhoods of width $1$ about the different boundary circles and base geodesics of funnels (if they occur) should be pairwise disjoint.
This is achieved by choosing them sufficiently short.
It is also understood that each such $\vf_P$ is extended by zero onto the rest of $S$.
Then the $-\chi(S)$ different $\vf_P$ are square integrable Lipschitz functions on $S$ which are pairwise $L^2$-orthogonal.

The area of each domain $N$ where a function $\vf_P$ decays is $\sinh(1)\ell$, where $\ell$ denotes the length of the corresponding closed boundary or base geodesic.
Moreover, the gradient of $\vf_P$ has norm $1$ on these domains $N$.
Therefore \[\int|\nabla\vf_P|^2=\sum|N|=\sinh(1)\sum\ell,\] where the sum is over the boundary components and funnels of $P$.
For the Rayleigh quotient of $\vf_P$, we obtain
\begin{align*}
  R(\vf_P) = \frac{\int|\nabla\vf_P|^2}{\int\vf_P^2} \le \frac{\sinh(1)\sum\ell}{2\pi-\sinh(1)\sum\ell}.
\end{align*}
Choosing the hyperbolic metric on $P$ such that the lengths $\ell$ are sufficiently small, the right hand side is less than $\ve$.

If the hyperbolic metric on $P$ has no cusps, the support of $\vf_P$ is compact.
If it has cusps, we modify $\vf_P$ along each cusp by having it decay linearly from $1$ to $0$ along an interval of length $1$ along the outgoing geodesics.
Then the Rayleigh quotient stays less than $\ve$ if this is done sufficiently far out.
We arrive at $-\chi(S)$ pairwise $L^2$-orthogonal Lipschitz functions with compact support which vanish along the boundary of $S$ and which have Rayleigh quotients less than $\ve$.
Now the essential spectrum of $S$ is contained in $[1/4,\infty)$, by \eqref{ess0}, and $\ve\in(0,1/4]$.
Hence the variational characterization of Dirichlet eigenvalues below the bottom of the essential spectrum implies that $S$ has at least $-\chi(S)$ Dirichlet eigenvalues less than $\ve$.
\end{proof}

\begin{rem}\label{ends}
In \cref{bus}, any two complimentary subsets $C$ and $H$ in the set of ends of $S$ may be chosen to consist of cusps and funnels, respectively. 
\end{rem}

\begin{rem}\label{laxp}
By the work of Lax and Phillips, an infinite hyperbolic hinge cannot carry a non-trivial square-integrable solution $\vf$ of the equation $\Delta\vf=\lambda\vf$ with $\lambda\ge1/4$; see \cite[Theorem 4.8]{LP}.
(Notice that \cite[Theorem 4.8]{LP} also applies in dimension two, see the last sentence in Section 4 of \cite{LP}.)
Hence in \cref{bus}, if $S$ has a funnel, then $S$ does not have eigenvalues in $[1/4,\infty)$.
\end{rem}

Next, we present an extension of Randol's \cref{iran} with an elementary geometric proof, partially motivated by Buser's argument.

\begin{thm}\label{ran}
Let $S$ be a complete and connected hyperbolic surface of finite area with closed geodesics as boundary circles (if $\partial S$ is not empty) and with a two-sided non-separating simple closed geodesic of length $\ell$ in the interior of $S$.
Let $n\ge1$, $\ve>0$, and
\begin{align*}
  k \ge \frac{\ell e^{\ell}}{2\sinh(\ell/2)\ve|S|}.
\end{align*}
If $S$ is not compact, assume also that $\ve\le1/4$.
Then $S$ has a cyclic hyperbolic covering of order $(k+2)n$ with at least $n$ Neumann eigenvalues in $[0,\ve)$.
\end{thm}

If $\partial S=\emptyset$, the Neumann eigenvalues are the usual eigenvalues of $S$.

\begin{figure}
\includegraphics[width=11cm]{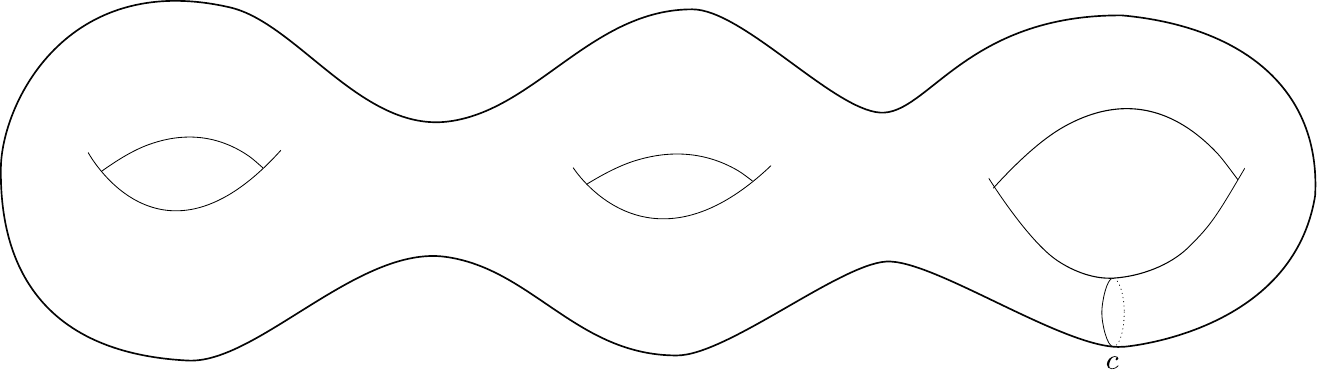}
\caption{Choosing the geodesic $c$}\label{figu_init}
\end{figure}

\begin{figure}
\includegraphics[width=11cm]{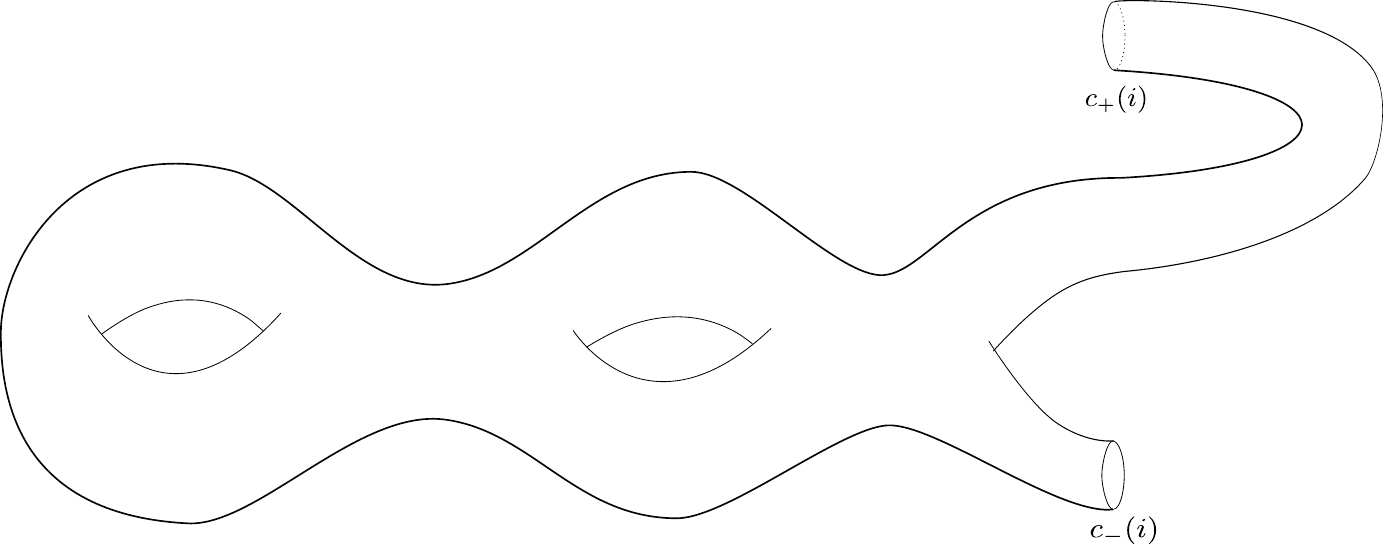}
\caption{A building block $T(i)$}\label{figu_build}
\end{figure}

\begin{figure}
\includegraphics[width=11cm]{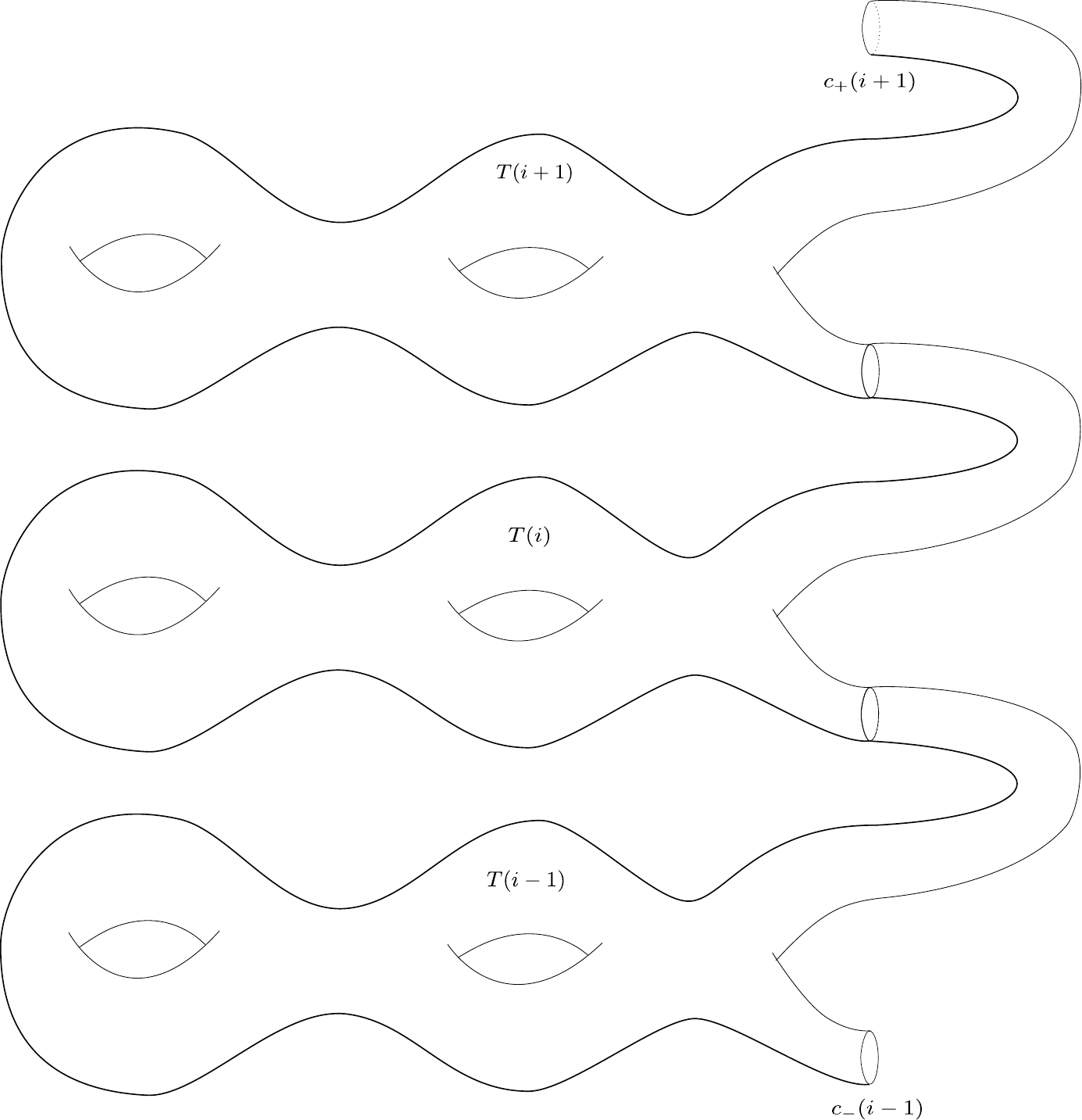}
\caption{A piece of the covering $\tilde S$}\label{figu_cov}
\end{figure}

\begin{proof}[Proof of \cref{ran}]
Let $c$ be a two-sided non-separating simple closed geodesic on $S$ of length $\ell=L(c)$ (see \cref{figu_init}).
Then $c$ has a tubular neighborhood $U$ which is isometric to $(-\rho,\rho)\times\R/\ell\Z$,
where $U$ is equipped with the Riemannian metric
\begin{align*}
  dr^2 + \cosh(r)^2ds^2
\end{align*}
and where $\rho>0$ is specified later. 
Cut $S$ along $c=\{r=0\}$ to obtain a connected surface $T$ with two boundary circles $c_-$ and $c_+$ and boundary collars \[C_-=(-\rho,0]\times\R/\ell\Z\quad\text{and}\quad C_+=[0,\rho)\times\R/\ell\Z\] containing them.
Let $\vf_-$ and $\vf_+$ be the Lipschitz functions on $T$ which are equal to $1$ on $T\setminus C_-$ and $T\setminus C_+$, vanish on $c_-$ and $c_+$, and are linear in $r\in(-\rho,0]$ and $r\in[0,\rho)$, respectively.
The Dirichlet integrals are
\begin{align*}
  \int_{T}|\nabla\vf_-|^2 = \int_{T}|\nabla\vf_+|^2 = \frac{\ell\sinh\rho}{\rho^2}
\end{align*}
since $|\nabla\vf_-|=1/\rho$ on $C_-$, $|\nabla\vf_+|=1/\rho$ on $C_+$, and $|C_-|=|C_+|=\ell\sinh\rho$.
Let $\ve>0$ and choose $k\ge1$ such that
\begin{align}\label{ksuf}
  k|S| = k|T| > \frac{2\ell\sinh\rho}{\rho^2\ve}.
\end{align}
Let $T(0),\dots,T(k+1)$ be $k+2$ copies of $T$ and attach $T(i)$ along $c_+(i)$ to $T(i+1)$ along $c_-(i+1)$,
for all $0\le i\le k$, to obtain a surface $R$ with two boundary circles and boundary collars $C_-(0)$ and $C_+(i+1)$ containing them (see Figures \ref{figu_build} and \ref{figu_cov}).
Let $\vf$ be the function on $R$ which is equal to $\vf_-$ on $T(0)$, to $\vf_+$ on $T(k+1)$, and to $1$ elsewhere.
Then the Rayleigh quotient of $\vf$ satisfies
\begin{align*}
  R(\vf) \le \frac{2\ell\sinh\rho}{\rho^2k|S|} < \ve.
\end{align*}
Now take $n$ copies $R_1,\dots,R_n$ of $R$, attach them naturally modulo $n$ and get $n$ copies $\vf_1,\dots,\vf_n$ of $\vf$, with pairwise disjoint supports (up to measure $0$), on the resulting closed surface $\tilde S$.
Here it is understood that each $\vf_i$ is extended by $0$ to all copies $R_j$ of $R$ with $j\ne i$.
Clearly, $\tilde S$ is a cyclic hyperbolic covering surface of $S$ of order $(k+2)n$ with $n$ non-vanishing Lipschitz functions $\vf_1,\dots,\vf_n$ with pairwise disjoint supports (up to measure $0$) and Rayleigh quotient $<\ve$.
In the case where $S$ is compact these functions immediately imply the existence of $n$ Neumann eigenvalues of $\tilde S$ in $[0,\ve)$, by the variational characterization of eigenvalues.

If $S$ is not compact, $\tilde S$ is still of finite area, and thus the ends of $\tilde S$ are cusps.
In particular, the essential spectrum of the Neumann extension $\Delta_N$ of the Laplacian on $\tilde S$ is contained in $[1/4,\infty)$,
by the characterization of the bottom of the essential spectrum in \eqref{ess}.
Now the supports of the above functions $\vf_1,\dots,\vf_n$ are not compact any more.
However, this can be remedied by cutting them off appropriately along the cusps of $\tilde S$ as in the proof of \cref{bus}.
Thus we obtain again the existence of $n$ Neumann eigenvalues of $\tilde S$ in $[0,\ve)$.

It remains to discuss the choice of $\rho$.
By \cref{coll}, the neighborhood of $c$ of width $\rho=\arsinh(1/\sinh(\ell/2))$ is of the form needed in the above argument, i.e., \eqref{ksuf} requires
\begin{align}\label{knee}
  k > \frac{2\ell}{\sinh(\ell/2)\arsinh(1/\sinh(\ell/2))^2\ve|S|}.
\end{align}
For $x>0$, we have
\begin{align*}
  \arsinh(1/\sinh x)
  &= \ln\frac{1+e^{-x}}{1-e^{-x}}
  = 2\artanh(e^{-x}) \\
  &= 2\big(e^{-x}+\frac{e^{-3x}}{3}+\frac{e^{-5x}}{5}+\cdots\big).
\end{align*}
Therefore the $\ell$-term on the right hand side of \eqref{knee} satisfies
\begin{align*}
  \frac{2\ell}{\sinh(\ell/2)\arsinh(1/\sinh(\ell/2))^2} 
  &
  = \frac{\ell}{2\sinh(\ell/2)\big(e^{-\ell/2}+e^{-3\ell/2}/3+\cdots\big)^2} \\
  &
  < \frac{\ell e^{\ell}}{2\sinh(\ell/2)}.
\end{align*}
Now the assertion follows from the first part of the proof. 
\end{proof}

\begin{rem}
A surface $S$ of finite type has two-sided non-separating simple closed curves if and only if its genus is positive.
\end{rem}

\begin{rem}
If $S$ is diffeomorphic to a closed orientable surface $S$ with at most one puncture, then a simple closed curve on $S$ is homologous to zero if and only if it is separating.
Thus the simple closed geodesics on $S$ in \cref{ran} are the homologically non-zero ones.
\end{rem}

\begin{thm}\label{rancor}
Let $\gamma,n\ge2$, $\ve>0$, and $k\ge2\ln(4\gamma-2)/\ve$.
Then, for any hyperbolic metric on $S_\gamma$, there is a cyclic hyperbolic cover $\tilde S\to S_\gamma$ of order $(k+2)n$ with at least $n$ eigenvalues in $[0,\ve)$.
\end{thm}

\begin{proof}
By the Gauss-Bonnet formula, the area of any hyperbolic metric on $S=S_\gamma$ is $2\pi(2\gamma-2)$.
Clearly, the systole of any hyperbolic metric on $S$ is twice its injectivity radius.
On the other hand, we have $\Sys^*(S)=\Sys(S)$, by \cref{parl}.
We conclude that, for a given hyperbolic metric on $S$, a shortest two-sided non-separating simple closed geodesic has length $\ell$ satisfying
\begin{align*}
  2\pi(\cosh(\ell/2)-1) = |B(\ell/2)| < |S| = 2\pi(2\gamma-2).
\end{align*}
That is, $\cosh(\ell/2)<2\gamma-1$ and, in particular, $\ell<2\ln(4\gamma-2)$.
Now the function $\ell e^{\ell}/2\sinh(\ell/2)$ is monotonically increasing.
Therefore we get that
\begin{align*}
  \frac{\ell e^\ell}{2\sinh(\ell/2)\ve|S|}
  &< \frac{(4\gamma-2)^2\ln(4\gamma-2)}{2\sinh(\ln(4\gamma-2))\ve\pi(2\gamma-2)} \\
  &< \frac{2(4\gamma-2)^2\ln(4\gamma-2)}{(4\gamma-2)\ve\pi(2\gamma-2)} \\
  &< \frac{6\ln(4\gamma-2)}{\ve\pi} < \frac{2\ln(4\gamma-2)}{\ve},
\end{align*}
where we use that $\ln(4\gamma-2)\ge\ln6>1$ and that $\sinh x>e^x/4$ for $x>1$.
Now the assertion follows from \cref{ran}.
\end{proof}

\begin{rem}
A corresponding result holds for hyperbolic surfaces of finite area with one cusp.
In general, the estimate on the shortest possible length of two-sided non-separating simple closed geodesics of complete hyperbolic geodesics will depend on the topology of the surface; see \cite{Pa}. 
\end{rem}

\section{Topology of nodal sets and small eigenvalues}
\label{secns}
For a surface $S$ and a smooth function $\vf$ on $S$, the set
\begin{equation*}
  Z_\vf = \{x \in S \mid \varphi(x) =0\}
\end{equation*}
is called the \emph{nodal set} of $\vf$.
Furthermore, if $S$ is Riemannian and $h$ a smooth function on $S$,
then a solution of the Schr\"odinger equation $(\Delta + h)\varphi=0$ is called an \emph{$h$-harmonic function}.
In \cite{Che}, S.\,Y. Cheng proved the following structure theorem for nodal sets of $h$-harmonic functions on $S$.

\begin{thm}[Cheng]\label{Cheng}
Let $S$ be a Riemannian surface and $h$ a smooth function on $S$.
Then any $h$-harmonic function $\vf$ on $S$ satisfies:
\begin{enumerate}
\item
The critical points of $\vf$ on $Z_\varphi$ are isolated.
\item
When the nodal lines meet, they form an equiangular system.
\item
The nodal lines consist of a number of $C^2$-immersed one-dimensional submanifolds.
In particular, if $S$ is closed, then $Z_\vf$ is a finite union of $C^2$-immersed circles.
\end{enumerate}
\end{thm}

Each connected component of $S\setminus Z_\vf$ is called a \emph{nodal domain} of $\vf$.
Using the above structure of nodal sets and the \emph{Courant nodal domain theorem} (see \cite{Che}),
Cheng then proved bounds on the multiplicity of the $i$-th eigenvalue in terms of $i$ and the Euler characteristic of the surface.
These methods for bounding multiplicities of eigenvalues have proved to be fruitful in general. 

\subsection{S\'{e}vennec's idea}\label{sevennec}

The multiplicity of the first eigenvalue gained more interest than the others (see \cite{C}, \cite{Se} and references therein).
(Note that $\lambda_0$ is simple by Courant's nodal domain theorem.)
In \cite{Se}, B.\,S\'{e}vennec took a leap of thoughts and obtained  a significant improvement of the then best known bound on the multiplicity of the first eigenvalue of closed surfaces with negative Euler characteristic.

\begin{thm}[S\'{e}vennec]\label{seven}
If $S$ is a closed Riemannian surface with $\chi(S)<0$, then the multiplicity of the first eigenvalue $\lambda_1$ of a Schr\"odinger operator $\Delta+h$ on $S$ is at most $5-\chi(S)$.
\end{thm}

The ideas in S\'{e}vennec's approach proved to be fruitful in the work of Otal \cite{Ot}, Otal-Rosas \cite{OR}, and our work in \cite{BMM1,BMM2}.
S\'{e}vennec started by proving a Borsuk-Ulam type theorem (see Lemmata 7 and 8 in \cite{Se}) which has the following consequence.

\begin{lem} \label{borsuk-ulam}
Let $\cup_{i=1}^k \mathcal{P}_i = P^d$ be a decomposition of the $d$-dimensional real projective space into $k$ subsets.
Assume that the characteristic class $\alpha$ of the standard covering map $\pi\colon S^d \to P^d$ satisfies $(\alpha|_{\mathcal{P}_i})^{\ell_i}=0$, for all $1\le i\le k$.
Then $d+1\le\ell_1+\dots+\ell_k$.
\end{lem}

By elliptic regularity, the eigenspace $E_1$ of $\lambda_1$ as in \cref{seven} is finite dimensional. 
Now consider some norm on $E_1$ (all are equivalent),
and, with respect to this norm, consider the unit sphere $S^d$ in $E_1$, 
where $d+1=\dim E_1$ is the multiplicity of $\lambda_1$. 
S\'{e}vennec's investigation of the Borsuk-Ulam theorem was motivated by the fact that each non-zero eigenfunction $\vf\in E_1$ has exactly two nodal domains, $\Omega_\vf^-=\{\vf<0\}$ and $ \Omega_\vf^+=\{\vf>0\}$, which gives a decomposition of $S^d$ into the strata
\begin{align*}
  \mathcal{S}_1 &= \{\vf \in S^d \mid b_1(\Omega_\vf^+) + b_1(\Omega_\vf^-) \le1 \}, \\
  \mathcal{S}_j &= \{\vf \in S^d \mid b_1(\Omega_\vf^+) + b_1(\Omega_\vf^-) = j \}, \; 1<j\le b_1(S),
\end{align*}
where $b_1$ indicates the first Betti number.
Cleary, each $\mathcal{S}_j$ is invariant under the antipodal map of $S^d$.
Discussing the properties of this decomposition of $P^d$ into the strata $\mathcal{P}_i=\pi(\mathcal{S}_i)$ covers a significant part of \cite{Se}.
The main results are $\ell_1=4$ and $\ell_j=1$ for $1<j\le b_1(S)$ (\cite[Theorem 9]{Se}).

\subsection{Otal's adaptation to small eigenvalues}\label{otal:multi}

In \cite{Ot}, Otal adapted this whole line of thoughts to bound the multiplicity of small eigenvalues on hyperbolic surfaces of finite area.
Recall that $1/4$ is bottom of the spectrum of the Laplacian on the hyperbolic plane $\mathbb{H}^2$.
If  $\Omega$ in $\mathbb{H}^2$ is a bounded domain with piecewise smooth boundary, $\lambda_0(\Omega)$ is the first Dirichlet eigenvalue of $\Omega$.
Hence, by domain monotonicity of the first Dirichlet eigenvalue, we get the strict inequality $\lambda_0(\Omega) > 1/4$.

\begin{thm}[Otal] \label{otal_mult}
For a complete hyperbolic surface $S$ of finite area, the multiplicity of an eigenvalue of $S$ in $(0,1/4]$ is at most $-\chi(S) -1$.
\end{thm}

\begin{rem}
In \cite{Ot}, Otal also proves a similar result on the multiplicity of \emph{cuspidal eigenvalues} of $S$ in $(0,1/4]$.
\end{rem}

Observe that the eigenvalues considered now need not be the first non-zero eigenvalue.
Hence S\'{e}vennec's ideas can not be applied directly. 
To remedy this, Otal starts with a key observation that provides a strong constraint on the topology of nodal sets and nodal domains of eigenfunctions with eigenvalue $\lambda\in(0,1/4]$.

\begin{lem}\label{incom_nodal}
Let $S$ be a closed hyperbolic surface and $\varphi$ be a non-trivial eigenfunction of $S$ with eigenvalue $\lambda\in(0,1/4]$.
Then the nodal set of $\vf$ is incompressible and any nodal domain of $\vf$ has negative Euler characteristic.
\end{lem}

Here we say that a subset  $G\subseteq S$ is \emph{incompressible} if each loop in $G$, that is homotopically trivial in $S$, is already homotopically trivial in $G$.

\begin{proof}[Sketch of proof of \cref{incom_nodal}]
Observe that, for any nodal domain $\Omega$ of $\vf$, we have $\lambda_0(\Omega) = \lambda$. 
This follows immediately from the observations that $(i)~ \vf|_\Omega$ satisfies the eigenvalue equation on $\Omega$, $(ii)~ \vf|_{\partial\Omega} = 0$, and $(iii)~ \vf$ has constant sign on $\Omega$.

Now let $D$ be a nodal domain of $\vf$ that is a disk. Then $\lambda_0(D) =\lambda$.
On the other hand, the universal covering $\pi\colon\tilde S\to S$ is trivial over $D$ and so
we can lift $D$ to a disk $\tilde D$ in $\tilde S=\mathbb{H}^2$.
In particular, $\tilde D$ is isometric to $D$ and hence $\lambda_0(\tilde D)= \lambda_0(D) =\lambda \le 1/4$.
This is a contradiction to what we found in the first paragraph of this subsection, namely that $\lambda_0(\tilde D)> 1/4$.

To finish the proof of the first part of the lemma one observes that, if a simple loop in $Z_\vf$ would be homotopically trivial in $S$, then it would bound a disk in $S$, by the Schoenflies theorem,
and then there would be a nodal domain of $\vf$ that would be a disk, again by  the Schoenflies theorem.
The remaining assertion, i.e., that no nodal domain is an annulus or a M\"obius band, can be proved by similar arguments.
One extra ingredient one needs is that any annulus or M\"obius band in $S$ can be lifted to a cyclic subcover $\hat S$ of $\tilde S$ and that, by a result of Brooks \cite{Br2}, the bottom of the spectrum satisfies $\lambda_0(\hat S)=\lambda_0(\tilde S)=1/4$.
\end{proof}

\begin{rem} 
In the appendix of \cite{BMM3}, we give a short proof of Brooks' result in the case of normal cyclic coverings
as needed in the above application.
\end{rem}

\begin{proof}[Sketch of proof of \cref{otal_mult}]
The basic strategy is very similar to \cite{Se}.
For simplicity, we assume that $S$ is closed and let $\lambda\in(0,1/4]$ be an eigenvalue of $S$.
Let $E_\lambda$ be the eigenspace of $\lambda$
and denote the multiplicity $\dim E_\lambda$ of $\lambda$ by $d+1$.
Now the idea is again to use \cref{borsuk-ulam} and decompose, in a first step, the unit sphere $S^d$ in $E_\lambda$ (with respect to some norm) into $-\chi(S) -1$ many strata, using the topology of $S\setminus Z_\vf$. 
Otal chose the strata as
\begin{equation*}
\mathcal{S}_i = \{\varphi \in S_\lambda \mid \chi(S \setminus Z_\varphi)= -i\}.
\end{equation*}
Using \cref{incom_nodal} and the Euler-Poincar\'{e} formula, one can easily deduce that, for any $\lambda$-eigenfunction $\vf$, one has $ \chi(S) \le \chi(S \setminus Z_{\vf}) \le -2$.
In particular, $\mathcal{S}_i=\emptyset$ for $i \neq 2, \cdots, -\chi(S).$
Hence the above stratification consists of at most $-\chi(S) -1$ non-empty strata. 
From the definition it is clear that $\mathcal{S}_i$ is invariant under the antipodal map of $S_\lambda$.
Hence to conclude the theorem, one needs to prove that the restriction of the covering $\pi\colon S^d\to P^d$ to each stratum $\mathcal{S}_i$ is trivial,
where $P^d$ denotes the projective space of $E_\lambda$.

The argument for this part relies on the following fact:
If $U, V$ are two disjoint subsurfaces of $S$ with piecewise smooth boundary and at least one of $U$ or $V$ has negative Euler characteristic, then there is no isotopy of $S$ that interchanges $U$ and $V$.

For $\varphi \in S^d$, in the same line as \cite{Se}, consider the decomposition $S \setminus Z_\varphi$ according to the sign of $\varphi,$ 
i.e., $S\setminus Z_\varphi = C^+_\varphi \cup C^-_\varphi$, where $C^+_\varphi=\{\varphi>0\}$  and $C^-_\varphi=\{\vf<0\}$. 
Then $C^\pm_\varphi$ is a subsurface of $S$ with piecewise smooth boundary, 
where the possible singularities of the boundary of $C^\pm_\varphi$ are described in \cref{Cheng}.

Observe that, for any $\psi \in S^d$ sufficiently close to $\vf$, we have $\chi(C^\pm_\varphi) \ge \chi(C^\pm_\psi)$.
If we further assume that $\vf, \psi \in \mathcal{S}_i$, then the last two inequalities are actually equalities.
It then follows that there is an isotopy of $S$ that sends $\chi(C^\pm_\varphi)$ to $\chi(C^\pm_\psi)$. 
Hence, by our earlier assertion on the existence of such isotopies and \cref{incom_nodal}, the connected component of $\mathcal{S}_i$ that contains $\vf$ can not contain $-\vf$. 
This implies the triviality of the restriction of the covering $\pi\colon S^d \to P^d$ to the stratum $\mathcal{S}_i$.
\end{proof}

\subsection{Otal-Rosas proof of Buser-Schmutz conjecture}\label{subsecor}

In \cite{S2}, Schmutz showed that any hyperbolic metric on the surface $S_2$ has at most $2$ eigenvalues $<1/4$, and he, and Buser in \cite{Bu1}, conjectured that any hyperbolic metric on the closed surface $S=S_g$ has at most $2g-2$ eigenvalues below $1/4$. 
Observe that the above result of Otal already implies this conjecture if all eigenvalues of $S$ in $(0,1/4]$ coincide.
Of course, in general, this will not be the case, and so one needs to do more to prove the conjecture.
The proof of an extended version, \cref{ior}, was finally achieved by Otal and Rosas in \cite{OR}. 

\begin{proof}[Sketch of proof of \cref{ior}]
Although the line of approach is very similar to those explained in  \S \ref{sevennec} and \S \ref{otal:multi}, there are several new difficulties that appear.
Consider now the vector space $E$ spanned by the finitely many eigenspaces $E_\lambda$ of $S$ with $\lambda\le\lambda_0(\tilde S)$.

To extend the ideas in \S \ref{otal:multi}, one needs an extension of Lemma \ref{incom_nodal}.
Since the functions that we are considering now are linear combinations of eigenfunctions, \cref{Cheng} is no longer available.
However, since the underlying Riemannian metric of $S$ is real analytic, its eigenfunctions are real analytic functions and, therefore, also any (finite) linear combination of them. 
Hence (by \cite{L}, as explained in Proposition 3 of \cite{OR}), the nodal set of any such linear combination has the structure of a locally finite graph.

A next and more serious difficulty in extending the ideas from \S \ref{otal:multi} is that Lemma \ref{incom_nodal}
may no longer be true for the nodal sets of arbitrary linear combinations of the eigenfunctions in $E$. 
For example, the nodal set $Z_\varphi$ of $\varphi$ may have components that are not incompressible.
(Note also that $E$ contains the constant functions so that the nodal set of $\vf\in E$ may be empty.)
To take care of this, just delete all those components of $Z_\varphi$ that are contained in a topological disk
to obtain the modified graph, $G_\vf\subseteq Z_\vf$.
Now $G_\varphi$ may still not be incompressible in $S$;
however, the components of $S\setminus G_\vf$ are.

\begin{lem}\label{nodom}
For any $\vf\in E$, at least one component of $S\setminus G_\vf$ has negative Euler characteristic.
\end{lem}

\begin{proof}
Let $\vf\in E$.
Then the Rayleigh quotient $R(\varphi)$ of $\varphi$ is at most $\lambda_0(\tilde S)$, by the definition of $E$.
On the other hand, if any component of $S\setminus G_\vf$ would be a disk or an annulus,
then the Rayleigh quotient $R(\varphi|_C)$ of $\varphi$ restricted to any such component $C$ would be strictly bigger than $\lambda_0(\tilde S)$, by the argument in the first paragraph of \S \ref{otal:multi} for disks and the argument at the end of the proof of \cref{incom_nodal} for annuli.
But then the Rayleigh quotient of $\varphi$ on all of $S$ would be strictly bigger than $\lambda_0(\tilde S)$,
a contradiction. 
\end{proof}

We let $Y_\vf$ be the union of all components of $S\setminus G_\vf$ with negative Euler characteristic.
Then $Y_\vf$ is not empty, thanks to \cref{nodom} above, and $\chi(Y_\vf)<0$.
We also have $\chi(S)\le\chi(Y_\vf)$ by the Mayer-Vietoris sequence and the incompressibility of the components of $S\setminus G_\vf$.
This argument requires some thought.

By definition, each component $C$ of $S\setminus G_\vf$ is a union of a nodal domain $\Omega$ of $\vf$ with a finite number of disks in $S$ enclosed by $\Omega$.
We say that $C$ is positive or negative if $\vf$ is positive or negative on $\Omega$
and let $Y^+_\vf$ and $Y^-_\vf$ be the union of the positive and negative components of $Y_\vf$, respectively.
Then $Y_\vf$ is the disjoint union of $Y^+_\vf$ and $Y^-_\vf$.

One final modification is necessary for these $Y^\pm_\varphi$.
Namely, if a component of $S\setminus Y^+_\varphi$ or $S \setminus Y^-_\varphi$ is an annulus,
then we attach that annulus to its neighbour components in $Y^+_\varphi$ or $Y^-_\varphi$, respectively,
to obtain new subsurfaces $X^+_\varphi\supseteq Y^+_\varphi$ and $X^-_\varphi\supseteq Y^-_\varphi$. 
Note that $\chi(X^\pm_\varphi)=\chi(Y^\pm_\varphi)$ so that, in particular,
\begin{equation*}
  \chi(S) \le \chi(X^+_\varphi) + \chi(X^-_\varphi) <0,
\end{equation*}
by what we said above.

Now we are ready to follow the approaches in \S \ref{sevennec} and \S \ref{otal:multi}.
As before, we consider the unit sphere $S^d$ in $E$ and the projective space $P^d$ of $E$,
where $\dim E = d+1$.
The strata of $S^d$ as in \cref{borsuk-ulam} are now 
\begin{equation*}
\mathcal{S}_i = \{ \vf \in S^d: \chi(X^+_\varphi) + \chi(X^-_\varphi) = -i \}.
\end{equation*}
In order to show the triviality of the restriction of the covering $\pi\colon S^d\to P^d$ to $\mathcal{S}_i \to \mathcal{P}_i=\pi(\mathcal{S}_i)$,
one argues that the isotopy type of the triples $(S, X^+_\varphi, X^-_\varphi)$ does not change under a small perturbation of $\varphi$ as long as the perturbation lies in the same stratum. 
The proof of this last fact follows a similar line as the one in the last part of the (sketch of the) proof of \cref{otal_mult}.
\end{proof}

\section{Small eigenvalues and analytic systole}
\label{secas}
The proof of \cref{incom_nodal} suggests the following definition, that first appeared (without the name) in \cite[cf.\,Equation 1.6]{BMM1}.

\begin{dfn}
The \emph{analytic systole} $\Lambda(S)$ of a Riemannian surface $S$ is defined by
\begin{equation}
\Lambda(S) = \inf_\Omega \lambda_0(\Omega),
\end{equation}
where $\Omega$ runs over all compact disks, annuli, and M\"obius bands in $S$ with smooth boundary.
\end{dfn}

By domain monotonicity, one can also define the analytic systole using only annuli and M\"obius bands.
However, from the proof of \cref{ana_sys_thm} it will become clear why it is convenient to include disks in the definition.

\subsection{The number of small eigenvalues} \label{secas_small}
To put the next result into perspective, we note that
\begin{equation}\label{broo}
  \Lambda(S)\ge\lambda_0(\tilde S),
\end{equation}
by arguments as in the proof of \cref{incom_nodal}.

\begin{thm}\label{small}
A complete and connected Riemannian surface $S$ of finite type with $\chi(S)<0$ has at most $-\chi(S)$ eigenvalues in $[0,\Lambda(S)]$, counted with multiplicity.
\end{thm}

For $S$ with non-empty boundary, we assume the boundary to be smooth and the result refers to Dirichlet eigenvalues. 

This result was first proved for closed surfaces in \cite[Theorem 1.7]{BMM1} and then generalized to surfaces of finite type in \cite{BMM2}. 
The proofs of these results follow the approach from \cite{OR}, and again there are new problems one has to face. 
We start by describing the proof in the compact case and explain the additional arguments needed to handle the non-compact case briefly afterwards.

\subsubsection{Closed surfaces}
To circumvent the possibly bad regularity properties of nodal sets of non-vanishing linear combinations $\vf$ of eigenfunctions, we consider \emph{approximate nodal sets} $Z_\varphi(\ve)$ instead, i.e.
\begin{equation}
Z_\varphi(\ve)=\{\varphi^2 \leq \ve\},
\end{equation}
and the connected components of their complements, the \emph{approximate nodal domains}.
By Sard's theorem, almost every $\ve>0$ is a regular value of $\vf^2$, and we will restrict to such $\ve$ from here on.
For each such $\ve$, $Z_\varphi(\ve)$ is a subsurface of $S$.

In general, there is no need at all that the inclusion $S \setminus Z_\varphi(\ve) \to S$ is incompressible.
In order to overcome this problem, we modify $Z_\varphi(\ve)$ as follows: 
We remove any component of $Z_\varphi(\ve)$ that is contained in a closed disk $D \subset S$.
The resulting subsurface $Z'_\varphi(\ve) \subset Z_\varphi(\ve)$ is called
\emph{derived approximate nodal set}.
By construction, the complement $Y_\varphi(\ve)=S\setminus Z'_\varphi(\ve)$ is incompressible in $S$.
Moreover, any component of $Y_\vf(\ve)$ is the union of an approximate nodal domain of $\vf$ with a finite number of disks enclosed by it.
In particular, we may again assign signs to the components of $Y_\vf(\ve)$ to get $Y_\vf(\ve)=Y_\vf^+(\ve)\cup Y_\vf^-(\ve)$ as a disjoint union.
(Cf.\,Lemma 2.5 in \cite{BMM1}).
Note that it may happen that one of $Y_\vf^+(\ve)$ or $Y_\vf^-(\ve)$ is empty;
for example, if $\vf$ is a positive constant, then $Y_\vf^-(\ve)=\emptyset$.

Similar to the argument in \cref{subsecor}, we restrict our attention to those components of $Y_\varphi(\ve)$ and $Y_\varphi^\pm(\ve)$ having negative Euler characteristic and write $X_\varphi(\ve)$ and $X_\varphi^\pm (\ve)$ for the union of these components, respectively
(now following the notation in \cite{BMM2}).

In a next step, we show that $X_\varphi(\ve)$ is not empty (cf.\,Lemma 2.6 in \cite{BMM1}).
Since we are working with approximate nodal sets, the argument from the proof of \cref{nodom} only applies in the case where the Rayleight quotient $R(\varphi)<\Lambda(S)$ and shows that $X_\varphi(\ve)$ is not empty, for all sufficiently small $\ve>0$.
If instead $R(\varphi)=\Lambda(S)$, we need to analyze the situation much more carefully.
It turns out that, in this case, $\varphi$ is an eigenfunction so that we may use \cref{Cheng}, which allows us to understand the topology of $Y_\varphi(\ve)$ much better.
In fact, it is just the complement of a tubular neighbourhood around the nodal set of $\vf$.
Therefore, \cref{incom_nodal} implies that any component of $Y_\varphi(\ve)$ has negative Euler characteristic, for all sufficiently small $\ve>0$.

The last modification of the sets $X_\varphi(\ve)$ is exactly as in \cite{OR}.
That is, if components of $S\setminus X_\varphi^\pm(\ve)$ are annuli or M\"obius bands,
then attach these annuli and M\"obius bands to $X_\varphi^\pm(\ve)$ to obtain subsurfaces $S_\vf^\pm(\ve)$.
(Cf.\,Lemma 2.8 in \cite{BMM1}).
The final key observation is that these modifications in combination with incompressibility imply that the isotopy type of the triples $(S,S_\varphi^+(\ve),S_\varphi^-(\ve))$ stabilizes for $\ve$ small. (Cf.\ Lemma 2.10 in \cite{BMM1}).

From this point on, we can invoke the arguments from \cite{OR} again, although \cite{BMM1} is slightly more elementary in some parts of the proof.
This is due to the fact that we can use the implicit function theorem as a tool (cf.\,Lemma 3.2 in \cite{BMM1}), since $\ve$ is always chosen to be a regular value of $\varphi$. 

\subsubsection{Non empty boundary}
The case of non-empty smooth boundary follows along the same lines using only one extra ingredient.
This is an extension of \cref{Cheng} to the case of Dirichlet boundary values \cite[Theorem 1.7]{BMM2}.
The proof of the latter follows by standard reflection techniques.

\subsubsection{Non-compact surfaces}
The proof in the non-compact case relies on another modification procedure, which is related to the asymptotic behavior of approximate nodal sets.
Besides using approximate nodal sets instead of nodal sets, we also truncate the sets $Y_\varphi(\ve)$.
The reason why we need to introduce this truncation procedure is that, for two functions  $\varphi,\psi$, the value $\ve$ can be regular for $\varphi$, but not for $\psi$, even if $\| \varphi - \psi\|_{C^1}$ is very small.
A first important ingredient is that, during truncation, there is no loss of the relevant topology.

Replacing negative Euler characteristic by a different criterion,
we say that a component $C$ of $Y_\varphi(\ve)$ is an \emph{$F_2$-component}
if $\pi_1(C)$ contains $F_2$, the free group on two generators.

Let $K \subset S$ be a (large) compact subsurface with smooth boundary, such that $S \setminus K$ consists of a finite union of infinite cylinders.

As a consequence of incompressibility and the topology of $S \setminus K$, we have that any $F_2$-component of $Y_\varphi(\ve)$ has to intersect $K$.
Moreover, there can be only finitely many $F_2$-components.
(Cf.\ Lemma 4.14 in \cite{BMM2}).

As above, we restrict our attention to the $F_2$-components and denote by
$X_\varphi(\ve)$ the union of all the $F_2$-components of $Y_\varphi(\ve)$.
Using the Schoenflies theorem and incompressibility, it is possible to show that $K$ can be perturbed (possibly quite significantly, depending on $\varphi$),
such that $X_\varphi^{\pm}(\ve) \cap K$ is homotopy equivalent to $X_\varphi^{\pm}(\ve)$ (cf.\,the discussion on p.\,1761 in \cite{BMM2}).
We then call a pair $(\ve,K)$ \emph{$\varphi$-regular} if $\ve$ is a regular value of $\varphi^2$ and $X_\varphi^{\pm}(\ve,K)=X_\varphi^{\pm}(\ve)\cap K$ is homotopy equivalent to $X_\varphi^{\pm}(\ve)$.

After modifying the $X_\varphi^{\pm}(\ve,K)$ to subsurfaces $S_\varphi^{\pm}(\ve,K)$ as above,
the final preparatory step is to show that the topological type of the triples $(S,S_\varphi^+(\ve,K),S_\varphi^-(\ve,K))$, where $(\ve,K)$ is $\varphi$-regular, stabilizes for sufficiently small $\ve>0$ and sufficiently large $K$.
(Cf.\ Lemma 4.21 in \cite{BMM2}).
We can then use the topological types of these triples to invoke the machinery we have used above, 
but have to face several additional technical difficulties.

\begin{rem}
It might be a bit surprising at first that this only very rough description of the asymptotic behavior of approximate nodal domains is sufficient for our purposes.
This is remarkable, in particular, when compared to the proof of \cite[Th{\'e}or{\`e}me 2]{OR}, which uses the rather explicit description of eigenfunctions in cusps coming from separation of variables.
However, the main point for our topological arguments is that all topology of derived approximate nodal sets can be detected within large compact sets.
\end{rem}

\subsection{Quantitative bounds for the analytic systole} \label{secas_quant}
\cref{small} raises the natural problem of finding estimates for $\Lambda(S)$ in terms of other geometric quantities.

Our first result in this direction generalizes the main result of the third author in \cite{Mo},
which asserts that a hyperbolic metric on the closed surface $S_\gamma$ of genus $\gamma \ge2$
has at most $2\gamma-2$ eigenvalues $\le1/4+\delta$, where
\begin{equation*}
  \delta = \min\{\pi/|S|,\sys(S)^2/|S|^2\}.
\end{equation*}
Here $|S|$ denotes the area of $S$ and $\sys(S)$, the \emph{systole} of $S$,
is defined to be the minimal possible length of an essential closed curve in $S$.

\begin{thm}[Theorem 1.6 in \cite{BMM3}]\label{main2}
For a closed Riemannian surface $S$ with curvature $K\le\kappa\le0$, we have
\begin{equation*}
  \Lambda(S) \ge -\frac{\kappa}4 + \frac{\sys(S)^2}{|S|^2}.
\end{equation*}
\end{thm}

The proof relies on the Cheeger inequality for subsurfaces $F$ of $S$ with non-empty boundary,
\begin{equation*}
\lambda_0(F) \geq h(F)^2/4.
\end{equation*}
Recall that the \emph{Cheeger constant} is defined by
\begin{equation*}
h(F)=\inf \frac{\ell(\partial \Omega)}{|\Omega|},
\end{equation*}
where the infimum is taken over all subdomains $\Omega \subset \mathring F$ with smooth boundary and $\ell$ indicates length.

If $F$ is a disk, an annulus, or a M\"obius band, the assumed curvature bounds allow to apply isoperimetric inequalites (cf.\ \cite[Corollary 2.2]{BMM3}) giving the corresponding lower bound for the Cheeger constant.

One may view \cref{main2} also as an upper bound on the systole
in terms of a curvature bound and $\Lambda(S).$
Together with our next result, this explains the name \emph{analytic systole}.

For a closed Riemannian surface $S$, we say that a closed geodesic $c$ of $S$
is a \emph{systolic geodesic} if it is essential with length $L(c)=\sys(S)$.
Clearly, systolic geodesics are simple.

\begin{thm}[Theorem 1.8 in \cite{BMM3}]\label{esansy}
If $S$ is a closed Riemannian surface with $\chi(S)<0$ and curvature $K\ge-1$, then
\begin{equation*}
  \Lambda(S) \le \frac{1}{4} + \frac{4\pi^2}{w^2},
\end{equation*}
where $w=w(\sys(S))=\arsinh(1/\sinh(\sys(S)))$.
\end{thm}

Under the assumed curvature bound, \cref{coll} gives a cylindrical neighbourhood $T$ of a systolic geodesic of width at least $w$.
Applying Cheng's eigenvalue comparison to geodesic balls $B(x,r) \subset T$ with $r<w$ gives the result.

The combination of \cref{main2} and \cref{esansy} in the case of hyperbolic metrics 
is probably a bit more enlightening than the general case.

\begin{cor}\label{syshype}
For closed hyperbolic surfaces, we have
\begin{equation*}
  \frac{1}4 + \frac{\sys(S)^2}{4\pi^2\chi(S)^2}
  \le \Lambda(S)
  \le \frac{1}4 + \frac{4\pi^2}{w^2}
\end{equation*}
with $w=w(\sys(S))$ as in \cref{esansy}.
\end{cor}

\subsection{Qualitative bounds for the analytic systole} \label{secas_qual}
Our main qualitative result concerning the analytic systole of compact surfaces is as follows.

\begin{thm}[Theorem 1.1 in \cite{BMM3}] \label{ana_sys_thm}
If $S$ is a compact and connected Riemannian surface whose fundamental group is not cyclic, then $\Lambda(S)>\lambda_0(\tilde S)$.
\end{thm}

In \cite[Theorem 1.2]{BMM3}, we characterize the inequality $\Lambda(S)>\lambda_0(\tilde S)$
also for complete Riemannian surfaces $S$ of finite type.
The proof is quite involved.
To keep the arguments a bit easier, we will only discuss the inequality for compact surfaces
as stated in \cref{ana_sys_thm}.
This reduces the technical difficulties significantly but still contains most of the main ideas.

\subsubsection{The case $\chi(S)=0$.} (Cf.\,\cite[Section 4]{BMM3}.)
This case is much easier than the general case, since it follows immediately from \cref{main2}, when combined with a result of Brooks.
\begin{proof}
If $\chi(S)=0$, then $S$ is a torus or a Klein bottle.
Therefore, $\pi_1(S)$ is amenable and \cite{Br2} (or \cite{BMP}) implies that $\lambda_0(\tilde S)=0$.

Since $S$ is a torus or a Klein bottle, there is a flat metric $h$ conformal to the initial metric $g$.
Clearly, $g$ and $h$ are $\alpha$-quasiisometric for some $\alpha>1$, thus
\begin{equation*}
\Lambda(S,g)\geq \alpha^{-1} \Lambda(S,h).
\end{equation*}
Furthermore, we have
\begin{equation*}
\Lambda(S,h)\geq \frac{\sys(S,h)^2}{|S|_h^2}>0,
\end{equation*}
by \cref{main2}.
Hence $\Lambda(S,g)>0=\lambda_0(\tilde S)$ as asserted.
\end{proof}

\subsubsection{Inradius estimate}
For a domain $\Omega\subset S$ with piecewise smooth boundary, we call the unique positive, $L^2$-normalized Dirichlet eigenfunction of $\Omega$ the \emph{ground state}.
Recall also that the \emph{inradius} of $\Omega$ is defined to be
\begin{equation*}
\inrad(\Omega)=\sup \{r>0 \ | \ B(x,r) \subset \Omega \ \text{for some}\ x \in \Omega \}.
\end{equation*}
The first preliminary result we need in the remaining discussion is an inradius estimate for superlevel sets of ground states.

\begin{lem}[cf.\,Lemma 6.4 in \cite{BMM3}]\label{cheet}
There are constants $\rho, \ve_0>0$, such that for the ground state $\varphi$ of any compact disk, annulus, or M\"obius band $F$ in $S$ with smooth boundary and $\lambda_0(F)\leq \Lambda(S)+1,$ we have
\begin{equation*}
\inrad(\{\varphi^2 \geq \ve_0\}) \geq \rho.
\end{equation*}
\end{lem}

\begin{proof}[Sketch of proof]
The proof is based on isoperimetric inequalities for such domains (cf.\,\cite[Corollary 2.2]{BMM3}) and the monotonicity of the topology of superlevel sets (cf.\,\cite[Proposition 5.2]{BMM3}.
The used isoperimetric inequalities are suitable reformulations of classical ones.
The monotonicity of the topology of superlevel sets is a direct consequence of the maximum principle:
Any simple contractible loop $\gamma \subset F$ that is contained in $F_t=\{\varphi^2 \geq t\}$ has to be contractible in $F_t$.
Otherwise $\varphi$ would have a local minimum in the disk filling $\gamma$, which
contradicts the maximum principle.

Given this observation, one can invoke the coarea formula exactly as in the proof of the Cheeger inequality.
Instead of estimating from below by the Cheeger constant, the monotonicity of the topology allows us
 to use the isoperimetric inequalities, which give much more precise information.
\end{proof}
 
Let us write $\Lambda_D(S)=\inf_\Omega \lambda_0(\Omega),$ where the infimum runs through
all closed disks with piecewise smooth boundary in $S.$
Similarly, we write $\Lambda_A(S)$ and $\Lambda_{M}(S)$ in the case of annuli and M\"obius bands.
Then
\begin{equation*}
  \Lambda(S) = \inf\{\Lambda_D(S),\Lambda_A(S)\Lambda_M(S)\},
\end{equation*}
and we treat the three terms on the right hand side consecutively.

\subsubsection{Disks: $\Lambda_D(S)>\lambda_0(\tilde S)$}
(Cf.\,\cite[Theorem 7.2]{BMM3}.)
The proof relies on \cref{cheet}.
We argue by contradiction and assume that we have a sequence $D_i$ of disks in $S$ with $\lambda_0(D_i) \to \lambda_0(\tilde S).$

By the compactness of $S$ and \cref{cheet}, we can choose a (not relabeled) subsequence such that 
\begin{equation*}
B(x,\rho/2) \subset \{\varphi_i^2 \geq \ve_0\}
\end{equation*}
for some fixed ball $B(x,\rho/2).$
We lift the disks $D_i$ to disks $\tilde D_i\subset\tilde S$ such that 
\begin{equation*}
B(\tilde x,\rho/2) \subset \{\tilde \varphi_i^2 \geq \ve_0\}
\end{equation*}
for some fixed point $\tilde x$ over $x$.
We lift the ground states $\vf_i$ of $D_i$ to the ground states $\tilde \vf_i$ on $\tilde D_i$
and extend $\tilde\vf_i$ by zero to $S\setminus\tilde D_i$.

By extracting further subsequences if necessary, we have weak convergence $\tilde \varphi_i \to \tilde \varphi$ in $W^{1,2}(\tilde S)$
and, by standard elliptic estimates, $\tilde \varphi_i \to \tilde \varphi$ in $C^\infty(B(\tilde x, \rho/2)).$
In particular, 
\begin{equation} \label{low_bd}
\tilde\varphi^2 \geq \ve_0 \ \ \text{on}\ \ B(\tilde x, \rho/2).
\end{equation}
This rules out the worst case scenario $\tilde \varphi = 0.$
Moreover, by estimating the Rayleigh quotients of $\tilde \varphi_i$ on carefully chosen balls $B(\tilde x, r)$ with $r \to \infty$,
\eqref{low_bd} allows us to show that $R(\tilde \varphi)\leq \lambda_0(\tilde S).$

The idea behind this step is as follows: If $\supp \tilde \varphi_i \subset K$ for
some compact subset $K \subset \tilde S$,
the compact Sobolev embedding applies and we are done.
So assume that the supports of the $\tilde \varphi_i$ leave every compact set eventually.
Now remember that $B(\tilde x,\rho/2) \subset \supp \tilde \varphi_i$ and $|\supp \tilde \varphi_i| \leq |S|$.
Therefore, the intersection of $\partial B(\tilde x,r) \cap \supp \tilde \varphi_i$
has to be very short for many large $r$ and uniformly for a subsequence of $i$; compare with Figure \ref{figu}.
The precise argument is actually slightly different, but more difficult to picture.
The bound on $|\supp \tilde \varphi_i|$ is replaced by the bound $\|\tilde\varphi_i\|_{W^{1,2}}\leq 1+2\Lambda(S)$ and a counting argument gives the existence of the radii $r$.
By construction, the boundary term 
$\int_{\partial B(\tilde x,r)}  \tilde \varphi_i \langle \nabla \tilde \varphi_i, \nu \rangle$
coming from integrating by parts on $B(\tilde x, r)$ will be small, which allows to get good control on the Rayleigh quotient.

\begin{figure}
\includegraphics[width=11cm]{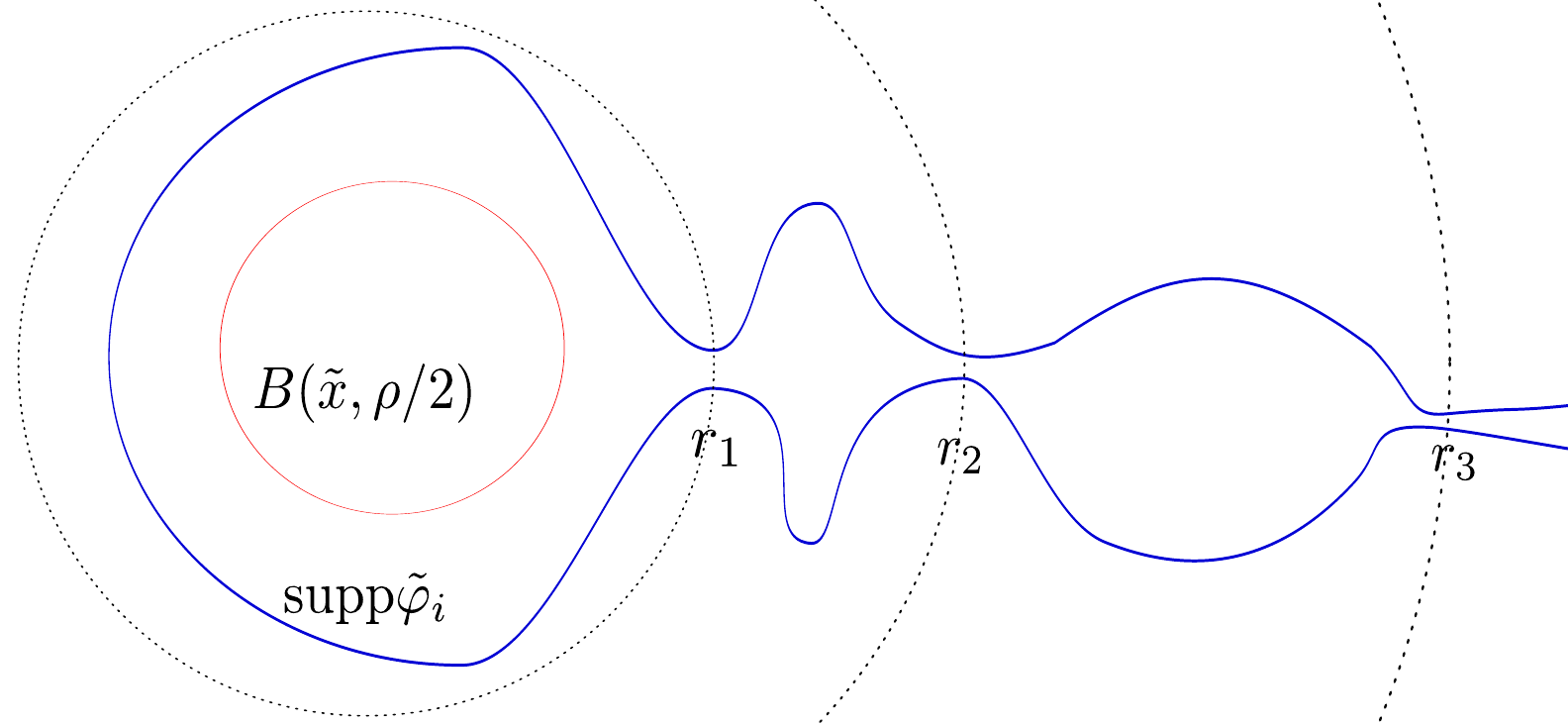}
\caption{The cutting radii $r_n$}\label{figu}
\end{figure}

The bound $R(\tilde \varphi)\leq \lambda_0(\tilde S)$ implies that $\tilde\vf$ is an $L^2$-eigenfunction of the Laplacian on $\tilde S$ with eigenvalue $\lambda_0(\tilde S)$.
Now $S$ is not a sphere, hence there is a point different from $\tilde x$ in the fiber over $x$.
By construction, $\tilde \varphi$ vanishes in the $\rho/2$ ball centered at this point, contradicting the unique continuation principle.
 
\subsubsection{Annuli 1: isotopy types}
For annuli we want to invoke a similar strategy.
Instead of lifting to the universal covering, we use cyclic subcovers.
A problem that we have to face is that this might be not the same subcover for different annuli in competition.

This is taken care of by the following comparison result.

\begin{lem}[cf.\ Lemma 5.3 in \cite{BMM3}]\label{lambdad}
If $F \subset S$ is a compact annulus, and $l$ denotes the length of a shortest non-contractible closed curve in $F$, then
\begin{equation*}
  \lambda_0(F)\ge \big\{1-\delta
  + 2\big(1-\frac1\delta\big)\frac{|F|}{\ell}\sqrt{\lambda_0(F)}\big\}\Lambda_D(S)
\end{equation*}
for all $0<\delta<1/2$.
\end{lem}

The proof of \cref{lambdad} relies on critical point theory for the ground state of $F$.

Since $S$ is compact, $|F|\leq |S|\leq C$.
For annuli with $\lambda_0(F) \leq \Lambda(S)+1$, the lemma above then implies that 
\begin{equation}
\lambda_0(F)\geq (1-2\delta)\Lambda_D(S)
\end{equation}
for $\ell$ large enough.
Since we already have $\Lambda_D(S)>\lambda_0(\tilde S),$ we can choose $\delta$ such that
$(1-2\delta)\Lambda_D(S)>(1+\delta)\lambda_0(\tilde S)$.
This then implies that only annuli with $\ell$ not too large can have $\lambda_0$ close to $\lambda_0(\tilde S).$
Finally, observe that there are only finitely many isotopy types with bounded $\ell$.
Thus, by extracting a subsequence if necessary, we may assume that all the annuli $A_i$ involved in a sequence with $\lambda_0(A_i)\to\lambda_0(\tilde S)$ belong to the same isotopy type.

\subsubsection{Annuli 2: $\Lambda_A(S)>\lambda_0(\tilde S)$} (Cf. \cite[Theorem 7.3]{BMM3}.)
Suppose that there is a sequence of annuli $A_i$ with $\lambda_0(A_i) \to \lambda_0(\tilde S)$.
By \cref{cheet}, we may assume that we have a fixed ball 
\begin{equation*}
B(x,\rho/2) \subset \{ \varphi_i^2 \geq \ve_0\}.
\end{equation*}
As explained above, \cref{lambdad} allows us to assume that all $A_i$ are isotopic.
Therefore, there is a cyclic subcover $\hat S$ of $\tilde S$ such that we can find lifts
$\hat A_i \subset \hat S$ of the $A_i.$
In contrast to the case of disks we are not free to choose these lifts, since the covering
$\hat S \to S$ is not normal. 
As before, we also lift the corresponding ground states $\varphi_i$ to functions $\hat \varphi_i$.
Then there is a sequence of balls $B(\hat x_i,\rho/2) \subset \{\hat\varphi_i^2 \geq \ve_0\}$.

We distinguish two cases.

\textsc{Case 1:} 
\textit{The sequence of $\hat x_i$ remains in a bounded set.}
\\
This allows us to repeat the argument, that we used for the case of disks.

\textsc{Case 2:} 
\textit{Up to passing to a subsequence, we have $\hat x_i \to \infty$.}
\\
This case is indicated in Figure \ref{fig}.
The metric on the cylinder is the pullback of the hyperbolic metric conformal to the
original metric on $S$.

\begin{figure}
\includegraphics[width=11cm]{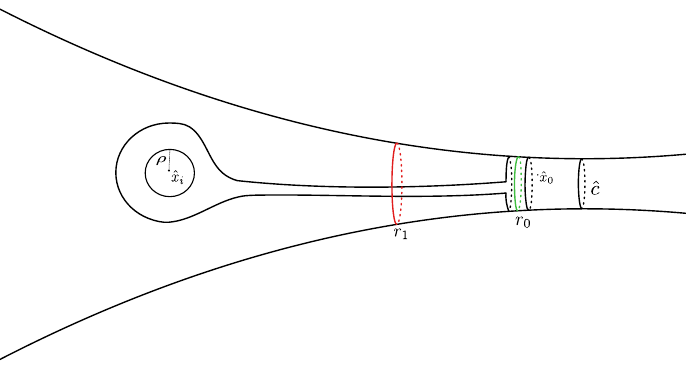}
\caption{The case $\hat x_i \to \infty$}\label{fig}
\end{figure}
 
Since the systole of the annuli $A_i$ is bounded, at least one boundary component of $\hat A_i$ has to lie within bounded distance to the closed geodesic $\hat c$ (indicated by the curve at distance $r_0$).
Since the area of $A_i$ is bounded, the entire second boundary component can not be too far away from $\hat c$ (indicated by the radius $r_1$).
Therefore, outside a fixed compact set $\hat A_i$ consists only of contractible components.

Let $\chi$ be a cut off that is $1$ on $B(\hat x_i,\rho)$ and vanishes in the compact set bounded by $r_1$.
By an integration by parts argument, we get that
\begin{equation*}
\int_{\hat S} |\nabla (\chi \hat \varphi_i)|^2 \leq \lambda_0(A_i) \int_{\hat S} | \chi \hat \varphi_i|^2 + C \delta,
\end{equation*}
where $\delta$ depends on $\chi$ and can be made small for $i$ large (since $\hat x_i \to \infty$).
Note that the inradius estimate implies in particular that
\begin{equation*}
\int_{B(\hat x_i,\rho)} \hat \varphi_i^2 \geq \ve_0 |B(x,\rho)| \geq C \ve_0 \rho^2.
\end{equation*}
Therefore,
\begin{equation*}
\Lambda_D(S) \leq R(\chi \hat \varphi) \leq \lambda_0(A_i) + C  \ve_0^{-1} \rho^{-2} \delta,
\end{equation*}
which contradicts $\lambda_0(A_i) \to \lambda_0(\tilde S) < \Lambda_D(S)$ for $i$ large (and thus $\delta$ small).

\begin{rem}
It is very interesting to observe that the proof crucially relies in two different ways 
on showing $\Lambda_D(S)>\lambda_0(\tilde S)$ first.
On one hand, this allows us to control the number of isotopy classes of annuli in competition.
On the other hand, it is used to rule out the second case from above for the lifted annuli.
\end{rem}

\subsubsection{M\"obius bands}

Showing $\Lambda_M(S)>\lambda_0(\tilde S)$ is now trivial thanks to the hard work we have done up to this point.
Let $S_o\to S$ be the orientation covering of $S$.
Any M\"obius band $M\subset S$ lifts to an annulus $A \subset S_o$.
Then
\begin{equation*}
\lambda_0(M) \geq \lambda_0(A) \geq \Lambda_A(S_o) 
\end{equation*}
which implies
\begin{equation*}
\Lambda_M(S) \geq \Lambda_A(S_o) > \lambda_0(\tilde S).
\end{equation*}
This finishes the proof of \cref{ana_sys_thm}.




\begin{thebibliography}{HMZ2}

\bibitem{BMM1}
W.\,Ballmann, H.\,Matthiesen, S.\,Mondal,
Small eigenvalues of closed surfaces.
\emph{J. Differential Geom.} {\bf 103} (2016), no. 1, 1--13,
MR3488128, Zbl 1341.53066.

\bibitem{BMM2}
W.\,Ballmann, H.\,Matthiesen, S.\,Mondal,
Small eigenvalues of surfaces of finite type.
\emph{Compositio Math.} {\bf 153} (2017), no. 8, 1747--1768,
MR3705274, Zbl 06764725.

\bibitem{BMM3}
W.\,Ballmann, H.\,Matthiesen, S.\,Mondal,
On the analytic systole of Riemannian surfaces of finite type.
\emph{Geom. Funct. Analysis}, to appear.

\bibitem{BMP}
W.\,Ballmann, H.\,Matthiesen, P.\,Polymerakis,
On the bottom of spectra under coverings.
\emph{Math. Zeitschrift}, to appear.

\bibitem{Br2}
R.\,Brooks,
The bottom of the spectrum of a Riemannian covering.
\emph{J. Reine Angew. Math.} {\bf 357} (1985), 101--114,
MR783536, Zbl 0553.53027.

\bibitem{Bu1}
P.\,Buser,
Riemannsche Fl\"achen mit Eigenwerten in $(0,1/4)$.
\emph{Comment. Math. Helv.} {\bf 52} (1977), no. 1, 25--34,
MR0434961, Zbl 0348.53027.

\bibitem{Bu2}
P.\,Buser,
\emph{Geometry and spectra of compact Riemann surfaces.}
Reprint of the 1992 edition. Modern Birkh\"auser Classics.
Birkh\"auser, 2010. xvi+454 pp.,
MR2742784, Zbl 1239.32001.

\bibitem{Che}
S.\,Y.\,Cheng,
Eigenfunctions and nodal sets.
\emph{Comment. Math. Helv.} {\bf 51} (1976), no. 1, 43--55,
MR0397805, Zbl 0334.35022.

\bibitem{C} Colin de Verdi\'{e}re, 
\emph{Construction de laplaciens dont une partie finie du spectre est donn\'{e}e}
Ann. scient. \'{E}c. Norm. Sup. 20 (1987), 599-615,
MR0932800, Zbl 0636.58036. 

\bibitem{DPRS}
J.\,Dodziuk, T.\,Pignataro, B.\,Randol, D.\,Sullivan,
Estimating small eigenvalues of Riemann surfaces.
\emph{The legacy of Sonya Kovalevskaya}
(Cambridge, Mass., and Amherst, Mass., 1985),
93--121, Contemp. Math., 64, Amer. Math. Soc., Providence, RI, 1987,
MR0881458, Zbl 0607.58044.

\bibitem{HS}
P.\,D.\,Hislop and P.\,D.\,Sigal,
\emph{Introduction to spectral theory. With applications to Schr\"odinger operators.}
Applied Mathematical Sciences, 113. Springer-Verlag, New York, 1996. x+337 pp.,
MR1361167, Zbl 0855.47002.

\bibitem{Hu}
H.\,Huber,
Zur analytischen Theorie hyperbolischer Raumformen und Bewegungsgruppen. II.
\emph{Math. Ann.} {\bf 142} (1960/1961), 385--398,
MR0154980, Zbl 0094.05703.
 
\bibitem{LP}
P.\,D.\,Lax, R.\,S.\,Phillips,
The asymptotic distribution of lattice points in Euclidean and non-Euclidean spaces.
\emph{Toeplitz centennial} (Tel Aviv, 1981), pp. 365--375,
Operator Theory: Adv. Appl., 4, Birkh\"auser, Basel-Boston, Mass., 1982,
MR065029, Zbl 0497.52007.

\bibitem{L}
S.\,\L{}ojasiewicz, 
Sur le probl\'{e}me de la division.
\emph{Studia Math.} 18 (1959), 87--136,
MR 0107168, Zbl 0115.10203.

\bibitem{Mat}
H.\,Matthiesen,
\emph{The theorem of Otal and Rosas, $\lambda_{2g-2}>1/4$.}
Bachelor Thesis, University of Bonn, 2013. 29 pp.

\bibitem{McK}
H.\,P.\,McKean,
Selberg's trace formula as applied to a compact Riemann surface.
\emph{Comm. Pure Appl. Math.} {\bf 25} (1972), 225--246.
Correction in \emph{Comm. Pure Appl. Math.} {\bf 27} (1974), 134,
MR0473166. 

\bibitem{Mo}
S.\,Mondal,
Systole and $\lambda_{2g-2}$ of closed hyperbolic surfaces of genus $g$.
\emph{Enseign. Math.} {\bf 60} (2014) no. 2, 1--23,
MR3262432, Zbl 1303.30037.  

\bibitem{Ot}
J.-P.\,Otal,
Three topological properties of small eigenfunctions on hyperbolic surfaces.
\emph{Geometry and dynamics of groups and spaces} 
Progr. Math. {\bf 265}, Birkh\"auser, Basel, 2008, 685--695,
MR2402419, Zbl 1187.35145.

\bibitem{OR}
J.-P.\,Otal, E.\,Rosas,
Pour toute surface hyperbolique de genre $g$, $\lambda_{2g-2}>1/4$.
\emph{Duke Math. J.} {\bf 150} (2009), no. 1, 101--115,
MR2560109, Zbl 1179.30041.

\bibitem{Pa}
H.\,Parlier,
The homology systole of hyperbolic Riemann surfaces.
\emph{Geom. Dedicata} {\bf 157} (2012), 331--338,
MR2893492, Zbl 1246.53060. 
 
\bibitem{Ra}
B.\,Randol,
Small eigenvalues of the Laplace operator on compact Riemann surfaces.
\emph{Bull. Amer. Math. Soc.} {\bf 80} (1974), 996--1000.
MR0400316, Zbl 0317.30017. 

\bibitem{S1}
P.\,Schmutz,
Small eigenvalues on Y-pieces and on Riemann surfaces.
\emph{Comment. Math. Helv.} {\bf 65} (1990), no. 4, 603--614, MR1078100, Zbl 0752.53026.

\bibitem{S2}
P.\,Schmutz,
Small eigenvalues on Riemann surfaces of genus 2.
\emph{Invent. Math.} {\bf 106} (1991), no. 1, 121--138,
MR1123377, Zbl 0764.53035.
 
\bibitem{SWY}
R.\,Schoen,  S.\,Wolpert, S.-T.\,Yau,
Geometric bounds on the low eigenvalues of a compact surface.
\emph{Geometry of the Laplace operator}
Proc. Sympos. Pure Math., XXXVI, Amer. Math. Soc., Providence, R.I., 1980,
MR573440, Zbl 0446.58018. 

\bibitem{Se}
B.\,S\'evennec,
Multiplicity of the second Schr\"odinger eigenvalue on closed surfaces.
\emph{Math. Ann.} {\bf 324} (2002), no. 1, 195--211,
MR1931764, Zbl 1053.58014.

\bibitem{T1}
M.\,E.\,Taylor,
\emph{Partial differential equations I. Basic theory.}
Second edition. Applied Mathematical Sciences, 115. Springer, New York, 2011. xxii+654 pp.,
MR1395148, Zbl 0869.35002.

\bibitem{T2}
M.\,E.\,Taylor,
\emph{Partial differential equations. II. Qualitative studies of linear equations.}
Applied Mathematical Sciences, 116. Springer-Verlag, 1996. xxii+528 pp.,
MR1395149, Zbl 0869.35003.

\end{thebibliography}
\end{document}